\documentclass[11pt,colorinlistoftodos]{article}

\usepackage{blindtext}
\usepackage{geometry}

\usepackage[utf8]{inputenc} 
\usepackage[T1]{fontenc}    
\usepackage[hidelinks]{hyperref}
\usepackage{url}            
\usepackage{booktabs}       
\usepackage{amsfonts}       
\usepackage{amsmath}
\usepackage{amsthm}
\usepackage{amssymb}
\usepackage{nicefrac}       
\usepackage{microtype}      
\usepackage{cleveref}       
\usepackage{lipsum}         
\usepackage{graphicx}
\usepackage{doi}
\usepackage{color}
\usepackage{tikz-cd}
\usepackage{lmodern}
\usepackage{mathtools}
\usepackage[thinc]{esdiff}
\usepackage{cases}
\usepackage{empheq}

\usepackage{cleveref}
\usepackage{lipsum} 
\usepackage{todonotes}
\usepackage{orcidlink}
\usepackage{xcolor}
\usepackage{authblk}

\newcommand{\F}{\mathbb{F}}

\newtheorem{thm}{Theorem}[section]

\newtheorem{corollary}[thm]{Corollary}
\newtheorem{prop}[thm]{Proposition}

\newtheorem*{thm*}{Theorem}

\theoremstyle{definition}
\newtheorem{definition}[thm]{Definition}
\newtheorem{remark}{Remark}[section]
\newtheorem{ex}{Example}[thm]

\DeclareMathOperator{\ad}{ad}

\DeclareMathOperator{\Imm}{Im}

\DeclareMathOperator{\Zz}{Z}
\DeclareMathOperator{\Leib}{Leib}

\DeclareMathOperator{\spn}{span}

\usepackage[%
backend=bibtex,%
style=numeric-comp,%
sorting=nyt,%
hyperref,%
]{biblatex}
\title{Biderivations of complete Leibniz algebras}
\author[1]{Alfonso Di Bartolo\footnote{{\href{mailto:alfonso.dibartolo@unipa.it}{\texttt{alfonso.dibartolo@unipa.it}}}}}
\author[2]{Francesco Paolo Di Fatta\footnote{\href{mailto:francesco.difatta@studenti.unime.it}{\texttt{francesco.difatta@studenti.unime.it}}}}
\author[1]{Gianmarco La Rosa\footnote{{\href{mailto:gianmarco.larosa@unipa.it}{\texttt{gianmarco.larosa@unipa.it}}}}\footnote{Corresponding Author}}

\affil[1]{Dipartimento di Matematica e Informatica, \protect\\ Università degli Studi di Palermo,\protect\\ Via Archirafi 34, 90123 Palermo, Italy}
\affil[2]{Doctoral School on Mathematics and Computational Sciences,\protect\\
University of Messina, Catania and Palermo}

\date{}

\setuptodonotes{color=blue!30}

\definecolor{math}{rgb}{0.0, 0.8, 0.6}

\begin{document}
\maketitle

\begin{abstract}
	If one wishes to define a complete Leibniz algebra in such a way as to extend the notion of a complete Lie algebra, two distinct definitions can be found in the current literature.  
	Since biderivations on complete Lie algebras have already been studied, in order to extend those results and considering that Leibniz algebras are, among others, a natural generalisation of Lie algebras, we study here the biderivations of complete Leibniz algebras according to both definitions.  
	In each case, we provide necessary and sufficient conditions for a bilinear map to be a biderivation of a Leibniz algebra.  
	Finally, we analyse both symmetric and skew-symmetric biderivations, highlighting their structural properties.
\end{abstract}



\section{Introduction}\label{sec:introduction}

A \emph{left Leibniz algebra} $L$ is a vector space over a field $\mathbb{F}$ equipped with a $\mathbb{F}$-bilinear form \([-,-] \colon L \times L \to L\) satisfying the Leibniz rule, i.\ e.\, \begin{equation}\label{eq:left_diff_id}
[x,[y,z]]=[y,[x,z]]+[[x,y],z],
\end{equation}
for all \(x,y,z\in L\). Similary, $L$ is called \emph{right Leibniz algebra} if the bilinear form $[-,-]$ satisfies the Leibniz rule
\begin{equation}\label{eq:right_diff_id}
	[[x,y],z]]=[[x,z],y]+[x,[y,z]].
		\end{equation}
		All results concerning left Leibniz algebras can be directly translated into corresponding results for right Leibniz algebras (including definitions) by introducing a new product $\{-,-\}$ on the same vector space, defined by $
		\{x,y\} = [y,x].$ For this reason, and in order not to overburden the present exposition, from now on we shall restrict our attention to left Leibniz algebras, simply writing $L$ to denote a Leibniz algebra.
		Clearly, every Lie algebra is a particular case of a Leibniz algebra, where the Lie bracket is alternating. In this setting, the Jacobi identity follows directly from either \Cref{eq:left_diff_id} or \Cref{eq:right_diff_id}, together with the alternating assumption. 
		A \emph{derivation} of a Leibniz algebra \( L \) is a linear map \( D \colon L \to L \) that satisfies the Leibniz rule
		$D([x, y]) = [D(x), y] + [x, D(y)]$, for all $x, y \in L$.
		Derivations abstract the behaviour of differentiation operators acting on algebras of smooth or continuous functions over manifolds, capturing the notion of infinitesimal symmetries within the algebraic framework. This is particularly the case for Lie algebras, where derivations play a central role in understanding their structure and representations.  A fundamental example of a derivation in the context of Lie algebras is the \emph{adjoint map}. For any \( x \in L \), the adjoint map \( \ad_x \colon L \to L \) is defined by \( \ad_x(y) = [x, y] \) for all \( y \in L \). Derivations of \( L \) of the form \( \ad_x \), for some \( x \in L \), are called \emph{inner derivations}. Nevertheless, derivations of Leibniz algebras also play a fundamental role in the study of such algebras, providing key insights into their internal symmetries and structural properties.
		
		A class that plays a fundamental role in the theory of Lie algebras (and likewise in Leibniz algebras) is that of semisimple algebras. One of their generalisations is given by \emph{complete} Lie algebras, namely those algebras with trivial center and all derivations being inner. In contrast to the case of Lie algebras, in the literature there exist two definitions of complete Leibniz algebras: one introduced in 2020 in \cite{BOYLE2020172} by Boyle, Misra, and Stitzinger, and another in 2022 in \cite{Ayupov2022} by Ayupov, Khudoyberdiyev, and Shermatova. Both definitions extend the classical notion of completeness for Lie algebras, in the sense that when restricted to Lie algebras, they coincide with the usual definition.  As we shall see in more detail later, there are Leibniz algebras that are complete according to the first definition but not the second, and vice versa. 
		
		A natural generalisation of the notion of a derivation is that of a biderivation, a two-dimensional analogue of a derivation in the sense that it acts as a bilinear map in each argument. This concept originated in 1980 when Maksa in \cite{maksa1980remark} studied symmetric biadditive maps with nonnegative diagonalisation on Hilbert spaces. Since then, biderivations have been investigated in various algebraic contexts, including rings (\cite{BRESAR1995764,GHOSSEIRI2013250}), algebras (\cite{BENKOVIC20091587}), Lie algebras (\cite{Bresar2018,Wu2023,tang2018biderivations}) and Lie superalgebras (\cite{Fan2017,Yuan2018}). The contexts mentioned above represent only a few of the settings in which biderivations have been studied. Furthermore, the cited works are not intended to be an exhaustive list of the literature on the subject, but rather serve as illustrative examples.

		In this paper, after recalling some preliminary results on Leibniz algebras and biderivations (\Cref{sec:preliminaries}), we investigate biderivations of complete Leibniz algebras under both definitions, in order to compare these results with those obtained for complete Lie algebras in \cite{DIBARTOLOLAROSA_COMPLETE_BIDER} (\Cref{sec:biderivations_complete_Leib}). Therefore, what we obtain is an extension of the results in \cite{DIBARTOLOLAROSA_COMPLETE_BIDER}. Finally, in \Cref{sec:Symm_skresymm_bider} we analyse symmetric and skew-symmetric biderivations in terms of linear commuting and skew-commuting maps.
		
		\section{Preliminaries}\label{sec:preliminaries}
		
		In this section, we briefly recall some fundamental notions concerning both Leibniz algebras and biderivations, which will be used throughout the paper.  
		For a more detailed and comprehensive treatment of Leibniz algebras, we refer the reader to a standard reference on the topic, such as \cite{ayupov2019}.  
		As mentioned above, we recall that, unless otherwise specified, \(L\) denotes a left Leibniz algebra over a field of characteristic zero.
		  
		We begin by recalling the following remark.
		
		\begin{remark} \label{remark:antisimm_1_componente}
			We observe that \([[x,y],z]=-[[y,x],z]\), for all $x,y,z\in L$. Indeed
			\begin{align*}
				[[x,y],z] & =[x,[y,z]]-[y,[x,z]]    \\
				          & =-([y,[x,z]]-[x,[y,z]]) \\
				          & =-[[y,x],z].          
			\end{align*}
		\end{remark}
		
		The \emph{Leibniz kernel} of a Leibniz algebra $L$ is the space \(\Leib(L)=\spn\langle[x,x]\mid x \in L \rangle\). We recall here that $\Leib(L)$ is more than a subspace of $L$, indeed, is an ideal.
		
		\begin{prop}\label{prop:leib_is_ideal}
			\(Leib(L)\) is an ideal of \(L\).
		\end{prop}
		
		\begin{proof}
			\(Leib(L)\) is a vector space for construction. Furthermore, for every \(x,y \in L\), \([y,[x,x]]=[x,[y,x]]+[[y,x],x] \in \Leib(L)\), so \([L,\Leib(L)] \subseteq \Leib(L)\). Lastly, 
			\([[x,x],y]=-[[x,x],y]\) by \Cref{remark:antisimm_1_componente}, so \([[x,x],y]=0 \in \Leib(L)\), i.e \([\Leib(L),L] \subseteq \Leib(L)\).
		\end{proof}
		
		\begin{remark} \label{remark:xy+yx}
			We observe that \([x,y]+[y,x]=[x+y,x+y]-[x,x]-[y,y] \in Leib(L)\) for all \(x,y\) in \(L\).
		\end{remark}
		
		With these results, we can consider \(L/\Leib(L)\), which is manifestly a Lie algebra. Moreover, $\Leib(L)$ is the smallest ideal of $L$ such that \(L/\Leib(L)\) is a Lie algebra (for further details see \cite{ayupov2019}).
		The \emph{left center} of $L$ is defined as \(\Zz^l(L)= \big\{ x \in L\mid[x,y]=0, \, \forall \, y \in L\big\}\). The \emph{center} of $L$ is defined as \(\Zz(L)= \big\{ x \in L\mid [x,y]=[y,x]=0 ,\, \forall \, y \in L\big\}\).

		\begin{prop}
			\(\Zz^l(L)\) is an ideal of \(L\).
		\end{prop}
		
		\begin{remark} \label{Leib in centre}
			By the proof of \Cref{prop:leib_is_ideal}, we observe that \([Leib(L),L]=0\), so \(\Leib(L) \subseteq \Zz^l(L)\).
		\end{remark}
		
		The following is the definition of derivation for a Leibniz algebra. The reader will not be surprised to discover that it follows closely from that for Lie algebras (which in turn follows from the one given for algebras, and so on).
		
		\begin{definition}
			A linear map \(D \colon L \to L\) is a \emph{derivation} of $L$ if 
			\begin{equation*}
				D([x,y])=[D(x),y]+[x,D(y)]
			\end{equation*}
			
			for all \(x,y\in L\).
		\end{definition}
		
		\begin{ex}
			Let \(L\) be a left Leibniz algebra, and we consider the maps of left multiplications on \(L\), namely \(L_x\), with \(x\in L\), where \(L_x(y)=[x,y]\). By \Cref{eq:left_diff_id}, \(L_x\) is a derivation of $L$. These kinds of derivations are called \emph{inner}.
		\end{ex}
		
		\begin{remark} \label{invarianza Leib}
			By \Cref{remark:xy+yx} we have \(D([x,x])=[D(x),x]+[x,D(x)] \in \Leib(L)\), so \(D(\Leib(L) \subseteq \Leib(L)\). In other words, $\Leib(L)$ is a \emph{characteristic ideal} of $L$.
		\end{remark}
		
		We are now ready to introduce the definition of biderivation, exactly as it is given for rings, algebras, Lie algebras, and so on.
		
		\begin{definition}\label{def:biderivation}
			Let \(f \colon L \times L \to L\) a \(\F-\)bilinear map. The map \(f\) is a \emph{biderivation} of $L$ if, for all \(x\in L\),  \(f(x,-)\) and \(f(-,x)\) are derivations of $L$, i.e
			\begin{align*}      
				f([x,y],z) & =[x,f(y,z)]+[f(x,z),y] \\
				f(x,[y,z]) & =[y,f(x,z)]+[f(x,y),z] 
			\end{align*}
			    
			\noindent for all \(x,y,z\in L\).
		\end{definition}
		
		\begin{remark}
			
			With this remark, we aim to clarify matters before the section dedicated to the results of our work. In the literature on Leibniz algebras, a definition of biderivation already exists. It was introduced by J.-L. Loday in 1993 in \cite{loday1993}: let $D$ be a derivation of a Leibniz algebra $L$, and let $d$ be an anti-derivation of $L$, that is, a linear map on $L$ such that $d([x,y]) = [x,d(y)] - [y,d(x)]$, for any $x,y\in L$. Then a biderivation for $L$, in Loday's sense, is the pair $(D,d)$ such that $[x, d(y)] = [x,D(y)]$ for all $x,y\in L$. In \cite{casas2025}, the following definition of biderivation (as a bilinear map) is given, namely a map $B\colon L\times L\to L$ such that 
			\begin{align}
				B([x,y],z) & =[x,B(y,z)]-[y,B(x,z)]\label{eq:Loday_bider_Leib_1} \\
				B(x,[y,z]) & =[B(x,y),z]+[y,B(x,z)]\label{eq:Loday_bider_Leib_2} 
			\end{align}
			for all $x,y,z\in L$. In this case, by setting $D(x)=B(y,x)$ and $d(x)=B(x,z)$, one can speak of a biderivation in Loday's sense. However, substantial differences arise between the two concepts:
			\begin{itemize}
				\item The definition does not capture the spirit in which biderivations originally arose (see the introduction and \cite{maksa1980remark}), namely as bilinear maps that behave as a derivation in each argument.
				\item As mentioned in the previous point, \Cref{eq:Loday_bider_Leib_1} does not tell us that $B$ is a derivation in the first argument when $L$ is a Leibniz algebra. If $L$ were a Lie algebra, then yes, but at that point the pair of derivation and anti-derivation simply becomes a pair of derivations since, if the bracket considered is anti-symmetric (as in a Lie algebra), the definitions of derivation and anti-derivation coincide.
			\end{itemize}
			This remark therefore aims to emphasize that throughout the rest of the paper, by biderivations we will mean bilinear maps as in \Cref{def:biderivation}, and not those that are usually understood for Leibniz algebras.
			
		\end{remark}
		
		In \cite{dibartoloLEFTRIGHTBIDER2025}, the authors introduced the notions of \emph{left} and \emph{right biderivations} in order to capture the bilateral behaviour of these maps and, moreover, to define a Lie bracket on the space of biderivations. We recall their definitions below, adapted to the case of a Leibniz algebra.
		
		\begin{definition}\label{def:left_right_biderivation}
			Let $L$ be a Leibniz algebra.  
			A map $B\colon L\times L\to L$ is called a:
			\begin{itemize}
				\item[a)] \emph{left biderivation} of $L$ if, for all $x,y,z\in L$, $B(\cdot,y)$ is linear, and
				      \begin{equation}\label{eq:leftbider}
				      	B(x,[y,z]) = [B(x,y),z] + [y,B(x,z)];
				      \end{equation}
				      
				\item[b)] \emph{right biderivation} of $L$ if, for all $x,y,z\in L$, $B(x,\cdot)$ is linear, and
				      \begin{equation}\label{eq:rightbider}
				      	B([x,y],z) = [x,B(y,z)] + [B(x,z),y].
				      \end{equation}
			\end{itemize}
		\end{definition}
		
		Clearly, a map is a biderivation if and only if it is both a left biderivation and a right biderivation.
		
		\subsection{Complete Leibniz algebras}
		
		In this section, we provide a brief review of the concept of complete Leibniz algebras. We begin by recalling the definition of a complete Lie algebra. 
		
		The notion of a \emph{complete} Lie algebra was introduced by N.\ Jacobson in 1962 \cite{jacobson2013lie}: a Lie algebra $L$ is complete if it has trivial center, $\operatorname{Z}(L)=0$, and every derivation is inner, $\operatorname{Der}(L)=\operatorname{ad}(L)$. Recall that semisimple Lie algebras satisfy both conditions, hence they provide natural examples of complete Lie algebras. However, completeness does not imply semisimplicity. In \cite{angelopoulos}, E. Angelopoulos exhibited a family of \emph{sympathetic} Lie algebras, i.e., complete Lie algebras satisfying $[L,L]=L$, which fail to be semisimple. Among these, there is a minimal-dimensional example: a $35$-dimensional Lie algebra whose Levi factor is isomorphic to $\mathfrak{sl}(2)$. For Leibniz algebras, however, the matter is slightly more subtle. Suppose one wishes to extend the definition of complete Lie algebras to the Leibniz setting. In that case, it is necessary to formulate axioms which, when specialised to Lie algebras, recover the classical conditions. Two such approaches appear in the literature.
		
		In 2020, Boyle, Misra, and Stitzinger gave the following definition of a complete Leibniz algebra.
		
		\begin{definition}[\cite{BOYLE2020172}]\label{def:complete1}
			A Leibniz algebra $L$ is said to be \emph{complete} if 
			\begin{itemize}
				\item[a)] \(\Zz(L/\Leib(L))=0\).
				\item[b)] For every derivation $D$ of $L$, there exists \(x \in L\) such that \(\Imm(D-L_x) \subseteq \Leib(L)\).
			\end{itemize}
		\end{definition}
		
		For such Leibniz algebras, hence we have \(\Zz^l(L)/\Leib(L) \subseteq \Zz(L/\Leib(L))\), so \(\Zz^l(L) \subseteq \Leib(L)\). According to \Cref{Leib in centre}, we have \(\Zz^l(L)=\Leib(L)\). Moreover, when $L$ is a Lie algebra, \Cref{def:complete1} reduces to the classical definition of completeness, since $\Leib(L)=0$.
		
		Two years later, Ayupov, Khudoyberdiyev, and Shermatova proposed an alternative definition of completeness for Leibniz algebras.
		
		\begin{definition} \label{def:complete2}
			A Leibniz algebra $L$ is said to be \emph{complete} if 
			\begin{itemize}
				\item[a)] \(\Zz(L)=0\).
				\item[b)] All derivations of $L$ are inner.
			\end{itemize}
		\end{definition}
		
		This definition, like the previous one, reduces to the classical definition of complete Lie algebras when $L$ is a Lie algebra, thereby fulfilling the initial goal. Surprisingly, however, the two definitions are not equivalent. Indeed, there exist Leibniz algebras that are complete in the sense of \Cref{def:complete1} but not in the sense of \Cref{def:complete2}, and conversely. 
		
		\begin{ex}\label{ex:completeLeib_si1no2}
			As reported by the authors in \cite{Ayupov2022}, the Leibniz algebra $L=\langle e_1,e_2,\ldots,e_n,x,y\rangle$ whose non-zero brackets are defined as
			\begin{align*}
				[e_1,e_1] & =e_3, & [e_i,e_1] & =e_{i+1},\quad 3\leq i\leq n-1 \\
				[e_1,x]   & =e_1, & [x,e_1]   & =-e_1,                         \\
				[e_2,y]   & =e_2, & [e_i,x]   & =(i-1)e_i,\quad 3\leq i\leq n  
			\end{align*}
			is solvable, complete in the sense of \Cref{def:complete1} but not in the sense of \Cref{def:complete2}.
		\end{ex}
		
		Conversely, we may consider a semisimple Leibniz algebra that is not a Lie algebra.  
		By Corollary~4.5 in \cite{feldvoss2024semisimpleleibnizalgebrasi}, we know that every finite-dimensional semisimple non-Lie Leibniz algebra over a field of characteristic zero admits derivations that are not inner.  
		Therefore, such a Leibniz algebra may be complete according to \Cref{def:complete1}, but it cannot be complete in the sense of \Cref{def:complete2}. 
		
		In order to make this situation visually clearer, we present in \Cref{fig:completeLeib} a diagram illustrating the relationships described above.  
		Specifically, let $\mathcal{X}_1$ be the set of complete Leibniz algebras as in \Cref{def:complete1}, $\mathcal{X}_2$ be the set of complete Leibniz algebras as in \Cref{def:complete2}, $\mathcal{X}$ be the set of complete Lie algebras, $L$ be the Leibniz algebra of \Cref{ex:completeLeib_si1no2}, and $S$ be a generic non-Lie semisimple Leibniz algebra.
		
		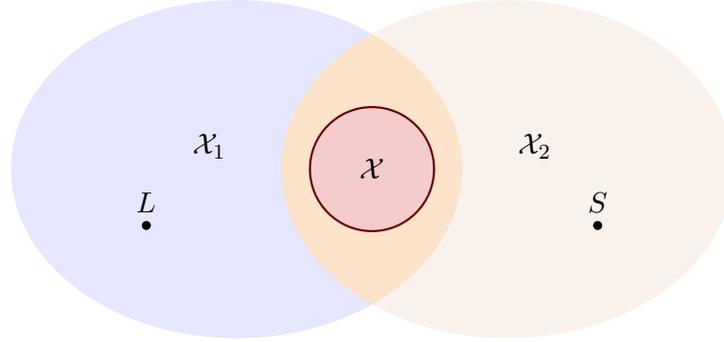
\begin{figure}
\begin{center}
\begin{tikzpicture}[scale=1.5]

  \draw[thick, fill=blue!10, draw=none] (-1.2,0) ellipse (2 and 1.5) node[above left] {$\mathcal{X}_1$};
  \filldraw[black] (-2,-0.5) circle (1pt);
  \node at (-2,-0.3) {$L$};;

  \draw[thick, fill=brown!10, draw=none] (1.2,0) ellipse (2 and 1.5) node[above right] {$\mathcal{X}_2$};
  \filldraw[black] (2,-0.5) circle (1pt);
  \node at (2,-0.3) {$S$};;
;

  \begin{scope}
    \clip (-1.2,0) ellipse (2 and 1.5);
    \fill[orange!30, opacity=0.6] (1.2,0) ellipse (2 and 1.5);
  \end{scope}

  \draw[thick, fill=red!80!black!20, draw=red!40!black] (0,0) circle [radius=.55];
  \node at (0,0) {$\mathcal{X}$};
\end{tikzpicture}
\caption{Diagrammatic representation of the relationships among the sets of complete Leibniz algebras under the two definitions of completeness.}
\label{fig:completeLeib}
\end{center}
\end{figure}

		\section{Biderivations of complete Leibniz algebras}\label{sec:biderivations_complete_Leib}
		
		This section is devoted to the study of biderivations on complete Leibniz algebras, according to both definitions available in the literature.  
		The aim is also to verify that, in both cases, the obtained theorems represent natural extensions of Proposition~4.1 in \cite{DIBARTOLOLAROSA_COMPLETE_BIDER}.
		
		\begin{thm}\label{thm:complete1}
			Let \(L\) be a complete Leibniz algebra according to \Cref{def:complete1}. A bilinear map \(f \colon L \times L \to L\) is a biderivation of $L$ if and only if there exist  \(\varphi,\psi \colon L \to L\) linear maps and \(p,q \colon L \times L \to \Leib(L)\) respectively right and left biderivations of \(L\) such that
			\begin{equation}\label{eq:complete1}
				f(x,y)=[\varphi (x),y]+p(x,y)=[\psi(y),x]+q(x,y)
			\end{equation}
		\end{thm}
		
		\begin{proof}
			
			The ``if'' direction is obvious. Let \(f\) be a biderivation and \(x \in L\). Then \(f(x,-)\) is a derivation, so there exists \(\tilde{\varphi}(x) \in L\) such that 
			\(f(x,y)-L_{\tilde{\varphi}(x)}(y) \in \Leib(L)\), i.e.
			\begin{equation*}
				f(x,y)=[\tilde{\varphi}(x),y]+\tilde{p}(x,y),
			\end{equation*}
			for some \(\tilde{p}(x,y) \in \Leib(L)\). For varying \(x, y \in L\), we define the maps \(\tilde{\varphi} \colon L \to L\) and \(\tilde{p} \colon L \times L \to \Leib(L)\).
			Now, fix a basis \(\{e_1, e_2, \ldots, e_n\}\) of \(L\), and let \(\varphi \colon L \to L\) be a linear map defined by \(\varphi(e_i) = \tilde{\varphi}(e_i)\).  
			Similarly, let \(p \colon L \times L \to \Leib(L)\) be a bilinear map such that \(p(e_i, e_j) = \tilde{p}(e_i, e_j)\).
			For every 
			\(x=\sum_{i=1}^n x_i e_i, y=\sum_{i=1}^n y_i e_i \in L\) we have:
			
			\begin{align*}
				f(x,y) & =f(\sum_{i=1}^n x_i e_i, \sum_{i=1}^n y_i e_i)                                                       \\
				       & =\sum_{i,j=1}^n x_i y_j f(e_i,e_j)                                                                   \\
				       & =\sum_{i,j=1}^n x_i y_j \big\{[\tilde{\varphi}(e_i),e_j]+\tilde{p}(e_i,e_j) \big\}                   \\
				       & =\sum_{i,j=1}^n x_i y_j \big\{[\varphi(e_i),e_j]+p(e_i,e_j) \big\}                                   \\
				       & =\sum_{i,j=1}^n x_i y_j [\varphi(e_i),e_j]+\sum_{i,j=1}^n x_i y_jp(e_i,e_j)                          \\
				       & =[\sum_{i=1}^n x_i\varphi(e_i), \sum_{j=1}^n y_j e_j] + p(\sum_{i=1}^n x_i e_i,\sum_{j=1}^n y_j e_j) \\
				       & =[\varphi(x),y]+p(x,y)                                                                               
			\end{align*}
			
			Furthermore, \(p(x,-)=f(x,-) - L_{\varphi(x)}\), so \(p(x,-)\) is a derivation for every \(x \in L\), i.e. \(p\) is a left biderivation.
			Similar computations with \(f(-,y)\) show that 
			\(f(x,y) = [\psi(y), x] + q(x,y)\), 
			where \(\psi \colon L \to L\) is a linear map and 
			\(q \colon L \times L \to \Leib(L)\) is a right biderivation.
			
		\end{proof}
		
		
		\begin{thm}\label{thm:complete2}
			Let \(L\) be a complete Leibniz algebra according to \Cref{def:complete2}. A bilinear map \(f \colon L \times L \to L\) is a biderivation if and only if there exist \(\varphi,\psi \colon L \to L\) linear maps such that, for any $x,y\in L$,
			\begin{equation}\label{eq:complete2}
				f(x,y)=[\varphi (x),y]=[\psi(y),x].
			\end{equation}  
		\end{thm}
		
		\begin{proof}
			The ``if'' direction is obvious. Let \(f\) be a biderivation and let \(x \in L\). Hence \(f(x,-)\) is a derivation, and then there exists \(\Tilde{\varphi}(x) \in L\) such that 
			\(f(x,y)=L_{\Tilde{\varphi}(x)}(y)\), i.e.,
			\(f(x,y)=[\Tilde{\varphi}(x),y]\) for every $x,y\in L$. As \(x\) varies on \(L\) we construct a map \(\Tilde{\varphi} \colon L \to L\) such that \(f(x,y)=[\Tilde{\varphi}(x),y]\). Now, we fix a basis \(\{e_1, e_2, \ldots, e_n\}\) of \(L\) and define a linear map \(\varphi \colon L \to L\) by setting \(\varphi(e_i) = \Tilde{\varphi}(e_i)\) for every \(1 \leq i \leq n\).
			Thus, for every \(x=\sum_{i=1}^n x_i e_i, y\in L\), we have 
			
			\begin{align*}
				f(x,y) & =f(\sum_{i=1}^n x_i e_i,y)                 \\
				       & =\sum_{i=1}^n x_i f(e_i,y)                 \\
				       & =\sum_{i=1}^n x_i [\Tilde{\varphi}(e_i),y] \\
				       & =\sum_{i=1}^n x_i [\varphi(e_i),y]         \\
				       & =[\sum_{i=1}^n x_i \varphi(e_i),y]         \\
				       & =[\varphi(\sum_{i=1}^n x_i e_i),y]         \\
				       & =[\varphi(x),y].                           
			\end{align*}

			Analogous computations yield the existence of a linear map \(\psi \colon L \to L\) satisfying \(f(x,y) = [\psi(y), x]\) for all \(x, y \in L\).
			
		\end{proof}
		
		\begin{remark}
			We observe that, in general, \(\tilde{\varphi} \neq \varphi\). Indeed, \(\tilde{\varphi}\) is not necessarily a linear map.  
			Even if \(f(x,y) = [\tilde{\varphi}(x), y]\) with \(f\) bilinear, we only obtain
			\begin{equation*}
				[\,\tilde{\varphi}(x+y) - \tilde{\varphi}(x) - \tilde{\varphi}(y),\, z\,]
				= f(x+y,z) - f(x,z) - f(y,z) = 0,
			\end{equation*}
			for all \(x,y,z \in L\). Hence,
			\[
				\tilde{\varphi}(x+y) - \tilde{\varphi}(x) - \tilde{\varphi}(y) \in \Zz^l(L),
			\]
			and analogously,
			\[
				\tilde{\varphi}(\lambda x) - \lambda\, \tilde{\varphi}(x) \in \Zz^l(L).
			\]
			This shows that \(\tilde{\varphi}\) is linear only modulo the left centre \(\Zz^l(L)\).

		\end{remark}
		
		If \(L\) is a Lie algebra, then it is a complete Leibniz algebra under both definitions. Moreover, \Cref{thm:complete1} and \Cref{thm:complete2} coincide with Proposition~4.1 in \cite{DIBARTOLOLAROSA_COMPLETE_BIDER}.
		
		\begin{corollary}
			Let \(L\) be a complete Lie algebra.  
			A bilinear map \(f \colon L \times L \to L\) is a biderivation if and only if there exist two linear maps \(\varphi, \psi \colon L \to L\) such that  
			\begin{equation*}
				f(x,y) = [\varphi(x), y] = [x, \psi(y)],
			\end{equation*}
			for all \(x, y \in L\).
		\end{corollary}
		
		\begin{proof}
			The "if" direction is obvious. For the "only if" direction, the case of \Cref{def:complete2} is immediate. In the case of \Cref{def:complete1}, since \(\text{Leib}(L) = 0\), we again obtain the same result.
		\end{proof}
		
		\section{Symmetric and Skew-Symmetric Biderivations
		 of Complete Leibniz Algebras}\label{sec:Symm_skresymm_bider}
		
		In this section, we study symmetric and skew-symmetric biderivations of complete Leibniz algebras. 
		In addition, we characterise the linear maps that determine these biderivations in the case of a complete Leibniz algebra, 
		and we investigate the behaviour of those biderivations arising from commuting and skew-commuting linear maps.
		
		A biderivation \(f \colon L \times L \to L\) is called \emph{symmetric} (resp. \emph{skew-symmetric}) if 
		\(f(x, y) = f(y, x)\) (resp. \(f(x, y) = -f(y, x)\)) for all \(x, y \in L\).
		As in the case of biderivations of an algebra over a field, every biderivation 
		\( f \colon L \times L \to L \) 
		of a Leibniz algebra \( L \) can be decomposed into the sum of its symmetric and antisymmetric parts:
		\[
			f = \frac{1}{2} f^+ + \frac{1}{2} f^-,
		\]
		where \( f^+ \) and \( f^- \) denote the symmetric and skew-symmetric components of \( f \), respectively, defined by
		\[
			f^+(x, y) = f(x, y) + f(y, x), \qquad
			f^-(x, y) = f(x, y) - f(y, x),
		\]
		for all \( x, y \in L \).

		It is straightforward to verify that both \(f^+\) and \(f^-\) are biderivations.
		We now analyse the case of symmetric/skew-symmetric for each definition of a complete Leibniz algebra.
		
		\begin{prop}
			Let \(L\) be a complete Leibniz algebra according to \Cref{def:complete1} and let \(f\) be a biderivation of \(L\). Then, referring to \Cref{eq:complete1}:
			\begin{itemize}
				\item[a)] If \(f\) is symmetric, then \([\sigma(x),y] \in \Leib(L)\), with \(\sigma=\varphi - \psi\).
				      
				\item[b)] If \(f\) is skew-symmetric, then \([\theta(x),y] \in \Leib(L)\), with \(\theta=\varphi + \psi\).
			\end{itemize}
			
		\end{prop}
		
		\begin{proof}
			\begin{itemize}
				    
				\item[a)] Since \( f(x, y) = f(y, x) \), we have  
				      \[
				      	[\varphi(x), y] + p(x, y) = [\psi(x), y] + q(y, x),
				      \]  
				      so  
				      \[
				      	[\varphi(x) - \psi(x), y] = -p(x, y) + q(y, x) \in \operatorname{Leib}(L),
				      \]  
				      because the images of \( p \) and \( q \) lie in the Leibniz kernel of $L$. 
				      
				\item[b)]  Using the identity \( f(x, y) = -f(y, x) \) and the same arguments, we obtain  
				      \[
				      	[\varphi(x) + \psi(x), y] = -\,p(x, y) - q(y, x) \in \operatorname{Leib}(L),
				      \]  
				      since the images of \(p\) and \(q\) lie in the Leibniz kernel of \(L\).
			\end{itemize}
			
		\end{proof}

		\begin{prop} \label{prop:sym-skewsym}
			Let \(L\) be a complete Leibniz algebra according to \Cref{def:complete2} and \(f\) a biderivation of \(L\). Then, referring to \Cref{eq:complete2}:
			
			\begin{itemize}
				\item[a)] If \(f\) is symmetric, then $\Imm\sigma\subseteq\Zz^l(L)$, where $\sigma=\varphi-\psi$. 
				      
				\item[b)] If \(f\) is skew-symmetric, then $\Imm\theta\subseteq\Zz^l(L)$, where $\theta=\varphi+\psi$. 
				      
			\end{itemize}
			
		\end{prop}
		
		\begin{proof}
			
			\begin{itemize}
				    
				\item[a)] Since \( f(x,y) = f(y,x) \), we have  
				      \begin{equation*}
				      	[\varphi(x), y] = [\psi(x), y],
				      \end{equation*}
				      which implies  
				      \begin{equation*}
				      	[\varphi(x) - \psi(x), y] = 0 \quad \text{for all } x,y \in L,
				      \end{equation*}
				      and hence  
				      \begin{equation*}
				      	\varphi(x) - \psi(x) = \sigma(x) \in \Zz^l(L) \quad \text{for all } x \in L.
				      \end{equation*}
				      
				\item[b)] Using the identity \( f(x, y) = -f(y, x) \) and similar arguments, we deduce that  
				      \begin{equation*}
				      	\varphi(x) + \psi(x) = \theta(x) \in \Zz^l(L) \quad \text{for all } x \in L.
				      \end{equation*}
			\end{itemize}
			
		\end{proof}
		
		Let \(L\) be a Leibniz algebra (not necessarily complete), and let \(V\) be an \(L\)-module. A linear map \(g \colon L \to V\) is called \emph{commuting} if 
		\[
			[g(x), x] = [x, g(x)] = 0 \quad \text{for every } x \in L,
		\] 
		(see \cite{casas2025}, Definition 4.3, for a reference).
		This implies that 
		\[
			[g(x), y] = -[g(y), x] \quad \text{for all } x, y \in L.
		\]  
		Although no definition of a skew-commuting map appears in the extant literature, this paper proposes one, motivated by the notion of commuting maps and by analogous situations in the context of Lie algebras (see, e.g., \cite{Bresar2018,chen2016,chengwangsunzhang2017,hanwangxia2016,wangyu2013}).
		Similarly, a linear map \(g \colon L \to V\) is called \emph{skew-commuting} if 
		\[
			[g(x), y] = [g(y), x] \quad \text{for all } x, y \in L.
		\]  
		More generally, the notion of commuting linear maps can be defined for a wider class of algebraic structures, such as rings and their modules, with some variations.
		To a commuting (or skew-commuting) linear map \(g\), we can associate the bilinear map 
		\(f \colon L \times L \to L\) defined by 
		\[
			f(x, y) = [g(x), y].
		\]
		The following propositions can be easily proved.
		
		\begin{prop} \label{prop:commuting}
			Let $g\colon L\to L$ be a commuting linear map. Then, the bilinear map \(f \colon L \times L \to L\) with \(f(x,y)=[g(x),y]\) is a skew-symmetric biderivation of \(L\).
		\end{prop}
		
		\begin{prop} \label{prop:skew-commuting}
			Let $g\colon L\to L$ be a skew-commuting linear map. Then, the bilinear map \(f \colon L \times L \to L\) with \(f(x,y)=[g(x),y]\) is a symmetric biderivation of \(L\).
		\end{prop}
		
		In particular, if \(L\) is a complete Lie algebra, the converse also holds: every symmetric biderivation is of the form described in \Cref{prop:commuting}, and every skew-symmetric biderivation is of the form described in \Cref{prop:skew-commuting} (as proved in \cite{DIBARTOLOLAROSA_COMPLETE_BIDER}).
		
		Following the definitions and the concluding propositions and observations of this section, it becomes natural to consider the case of complete Leibniz algebras. 
		Unfortunately, in the case of \Cref{def:complete1}, the "if and only if" condition does not hold, as the following examples illustrate.
		
		\begin{ex}
			Let \(L = \langle x, y \rangle\) be a Lie algebra with non-zero bracket \([x, y] = y\), and let \(V = \langle v \rangle\) be an \(L\)-module with actions defined by \(x \cdot v = v\) and \(y \cdot v = 0\). It is a module because
			\begin{equation*}
				0 = y \cdot v = [x, y] \cdot v = x \cdot (y \cdot v) - y \cdot (x \cdot v) = x \cdot 0 - y \cdot 0 = 0.
			\end{equation*}
			
			Next, we consider \(\tilde{L} = L \oplus V\) with the operation defined by
			\begin{equation*}
				[X + a, Y + b] = [X, Y] + X \cdot b
			\end{equation*}
			for all \(X, Y \in L\) and \(a, b \in V\).  
			This is a complete Leibniz algebra because \(\operatorname{Leib}(\tilde{L}) = V\) and \(\tilde{L} / \operatorname{Leib}(\tilde{L}) = L\) is a complete Lie algebra.
			
			Now, let \(\delta \colon \tilde{L} \to \tilde{L}\) be the linear map defined by
			\begin{equation*}
				\delta(x) = \delta(y) = 0, \qquad \delta(v) = v.
			\end{equation*}
			
			We claim that \(\delta\) is a derivation. Indeed, for all scalars \(\alpha, \beta, \gamma, \lambda, \sigma, \mu\),
			\begin{equation*}
				\delta([\alpha x + \beta y + \gamma v, \lambda x + \sigma y + \mu v]) 
				= \delta((\alpha\sigma - \beta\lambda)y + \alpha \mu v) 
				= \alpha \mu v,
			\end{equation*}
			and
			\begin{align*}
				  & [\delta(\alpha x + \beta y + \gamma v), \lambda x + \sigma y + \mu v] + [\alpha x + \beta y + \gamma v, \delta(\lambda x + \sigma y + \mu v)] \\
				  & = [\gamma v, \lambda x + \sigma y + \mu v] + [\alpha x + \beta y + \gamma v, \mu v]                                                           \\
				  & = \alpha \mu v,                                                                                                                               
			\end{align*}
			which shows that the derivation property holds.
			We define a symmetric bilinear map 
			\(F \colon \tilde{L} \times \tilde{L} \to \tilde{L}\) by
			\begin{equation*}
				\begin{aligned}
					F(x, x) & = 0, & F(x, y) & = 0, & F(x, v) & = 0, \\
					F(y, x) & = 0, & F(y, y) & = 0, & F(y, v) & = 0, \\
					F(v, x) & = 0, & F(v, y) & = 0, & F(v, v) & = v. 
				\end{aligned}
			\end{equation*}
			
			It is a biderivation because \(F(z, -) = F(-, z) = 0\) for all \(z \in L\), and 
			\(F(v, -) = F(-, v) = \delta\).
			Assume, for the sake of argument, that there exists a linear map 
			\(g \colon \tilde{L} \to \tilde{L}\) such that 
			\[
				F(z, t) = [g(z), t] \quad \text{for all } z, t \in \tilde{L}.
			\]
			Let \(g(v)=Ax+By+Cv\). Then we have:
			\begin{itemize}
				\item $0 = F(v, x) = [g(v), x] = -B y$, which implies $B = 0$.
				\item $0 = F(v, y) = [g(v), y] = A y$, which implies $A = 0$.
				\item $v = F(v, v) = [g(v), v] = 0$, which is absurd.
			\end{itemize}
			
		\end{ex}
		
		\begin{ex}
			Let \(L = \langle x, y \rangle\) be a Lie algebra with non-zero bracket \([x, y] = y\), and let \(V = \langle v, w \rangle\) be an \(L\)-module with actions defined by
			\[
				x \cdot v = v, \quad x \cdot w = w, \quad y \cdot v = y \cdot w = 0,
			\] 
			i.e., \(x \cdot\) acts as the identity on \(V\) and \(y \cdot\) acts as the zero map on \(V\).  
			
			Indeed, \(V\) is an \(L\)-module. We then construct \(\tilde{L} = L \oplus V\) with the operation
			\begin{equation*}
				[X + a, Y + b] = [X, Y] + X \cdot b
			\end{equation*}
			for all \(X, Y \in L\) and \(a, b \in V\).
			
			It is a complete Leibniz algebra (again, \(Leib(L)=V\) and \(\frac{\Tilde{L}}{Leib(L)}=L\)). 
			Let \(\delta \colon \tilde{L} \to \tilde{L}\) be a linear map such that \(\delta(V) \subseteq V\) and \(\delta(L) = 0\).  
			We claim that \(\delta\) is a derivation. Indeed, for all \(X, Y \in L\) and \(a, b \in V\), we have
			\begin{equation*}
				\delta([X + a, Y + b]) = \delta([X, Y] + X \cdot b) = \delta(X \cdot b),
			\end{equation*}
			and
			\begin{equation*}
				[\delta(X + a), Y + b] + [X + a, \delta(Y + b)] = [\delta(a), Y + b] + [X + a, \delta(b)] = X \cdot \delta(b),
			\end{equation*}
			since \(\delta(a) \in V\).  
			
			Moreover, if we write \(X = \alpha x + \beta y\), then
			\begin{equation*}
				\delta((\alpha x + \beta y) \cdot b) = \delta(\alpha b) = \alpha \delta(b),
			\end{equation*}
			and
			\begin{equation*}
				(\alpha x + \beta y) \cdot \delta(b) = \alpha \delta(b).
			\end{equation*}
			
			Hence, \(\delta\) satisfies the derivation property. Hence, every linear map \(\delta \colon \tilde{L} \to \tilde{L}\) satisfying \(\delta(L) = 0\) and \(\delta(V) \subseteq V\) is a derivation of \(\tilde{L}\).
			
			Let \(F \colon \tilde{L} \times \tilde{L} \to \tilde{L}\) be a skew-symmetric bilinear map defined by
			\begin{equation*}
				\begin{aligned}
					F(x,x) & = 0, & F(x,y) & = 0, & F(x,v) & = 0,  & F(x,w) & = 0, \\
					F(y,x) & = 0, & F(y,y) & = 0, & F(y,v) & = 0,  & F(y,w) & = 0, \\
					F(v,x) & = 0, & F(v,y) & = 0, & F(v,v) & = 0,  & F(v,w) & = v, \\
					F(w,x) & = 0, & F(w,y) & = 0, & F(w,v) & = -v, & F(w,w) & = 0. 
				\end{aligned}
			\end{equation*}
			
			It is a biderivation because \(F(z,-)\) and \(F(-,z)\) are derivations for all \(z \in \tilde{L}\).  
			
			Assume, for the sake of argument, that there exists a linear map \(g \colon \tilde{L} \to \tilde{L}\) such that 
			\begin{equation*}
			    F(z, t) = [g(z), t] \quad \text{for all } z, t \in \tilde{L}.
			\end{equation*}
			
			Let \(g(v) = A x + B y + C v + D w\). Then we have:
			\begin{itemize}
				\item $0 = F(v, x) = [g(v), x] = -B y$, which implies $B = 0$.
				\item $0 = F(v, y) = [g(v), y] = A y$, which implies $A = 0$.
				\item $v = F(v, w) = [g(v), w] = 0$, which is absurd.
			\end{itemize}
		\end{ex}
		
		Lastly, we consider \Cref{def:complete2}. If \(f\) is a biderivation of a complete Leibniz algebra \(L\), then there exist linear maps 
		\(\varphi, \psi \colon L \to L\) such that
		\begin{equation*}
			f(x, y) = [\varphi(x), y] = [\psi(y), x]
		\end{equation*}
		for all \(x, y \in L\).  
		
		Let \(W\) be a complementary subspace of \(\Zz^l(L)\), i.e., \(L = \Zz^l(L) \oplus W\) as a vector space, and let 
		\(\pi \colon L \to W\) denote the corresponding projection map.
		We have 
		\begin{equation*}
			[\pi(\varphi(x)),y]=[\varphi(x),y]=[\psi(y),x]=[\pi(\psi(y)),x]
		\end{equation*}
		
		By \Cref{prop:sym-skewsym}, if \(f\) is symmetric, then 
		\(\pi(\varphi(x)) - \pi(\psi(x)) \in W \cap \Zz^l(L)\), so \(\pi \varphi = \pi \psi\).  
		If we set \(g = \pi \varphi\), we have 
		\begin{equation*}
		    f(x, y) = [\pi(\varphi(x)), y] = [g(x), y].
		\end{equation*} 
		
		Similarly, if \(f\) is skew-symmetric, then \(\pi \varphi = -\pi \psi\).  
		Setting \(g = \pi \varphi\) again, we obtain 
		\begin{equation*}
		    f(x, y) = [\pi(\varphi(x)), y] = [g(x), y].
		\end{equation*} 
		In other words, we have the following results.
		
		\begin{prop} 
			Let \(L\) be a complete Leibniz algebra in the sense of \Cref{def:complete2}.  
			A symmetric bilinear map \(f \colon L \times L \to L\) is a biderivation if and only if there exists a skew-commuting linear map 
			\(g \colon L \to L\) such that 
			\begin{equation*}
				f(x, y) = [g(x), y] \quad \text{for all } x, y \in L.
			\end{equation*}
		\end{prop}

		\begin{prop} 
			Let \(L\) be a complete Leibniz algebra in the sense of \Cref{def:complete2}.  
			A skew-symmetric bilinear map \(f \colon L \times L \to L\) is a biderivation if and only if there exists a commuting linear map 
			\(g \colon L \to L\) such that 
			\begin{equation*}
				f(x, y) = [g(x), y] \quad \text{for all } x, y \in L.
			\end{equation*}
		\end{prop}
		
		\section*{Acknowledgments and Funding}
		
		Alfonso Di Bartolo was supported by the University of Palermo (FFR2024, UNIPA BsD).  Gianmarco La Rosa was supported by "Sustainability Decision Framework (SDF)" Research Project --  CUP B79J23000540005 -- Grant Assignment Decree No. 5486 adopted on 2023-08-04. The first and third authors were also supported by the “National Group for Algebraic and Geometric Structures, and Their Applications” (GNSAGA — INdAM).
		
		\section*{Conflict of Interest}
		
		The authors have no conflict of interest to declare that is relevant to this article.
		
		\section*{Data Availability}
		Data sharing is not applicable to this article as no datasets were generated or analysed during the current study.
		
		\section*{ORCID}
		
		Alfonso Di Bartolo \orcidlink{0000-0001-5619-2644} \href{https://orcid.org/0000-0001-5619-2644}{0000-0001-5619-2644}\\
		Francesco Di Fatta \orcidlink{0009-0008-6197-6126} \href{https://orcid.org/0009-0008-6197-6126}{0009-0008-6197-6126}\\
		Gianmarco La Rosa \orcidlink{0000-0003-1047-5993} \href{https://orcid.org/0000-0003-1047-5993}{0000-0003-1047-5993}
		
		\printbibliography
\end{document}